\documentclass{amsart}
\usepackage{amsmath,amsfonts,euscript,amscd,amsthm,amssymb,upref,graphics}
\usepackage{mathrsfs, amsmath}

\newcommand{\Escr}{\mathscr{E}}
\newcommand{\Kscr}{\mathscr{K}}
\usepackage[all]{xy}

\usepackage{ulem}

\usepackage{latexsym}
\usepackage{amsfonts,amssymb} 
\usepackage{graphicx}
\usepackage{tikz-cd}
\usepackage{caption}
\usepackage{subcaption}
\usepackage{enumitem}

\newcommand{\field}[1]{\mathbb{#1}}
\newcommand{\A}{\field{A}}
\newcommand{\C}{\field{C}}

\newcommand{\G}{\field{G}}

\newcommand{\N}{\field{N}}
\newcommand{\PP}{\field{P}}
\newcommand{\Q}{\field{Q}}

\newcommand{\Z}{\field{Z}}

\newcommand{\krn}{{\rm ker}\,}

\newcommand{\ggoth}{{\ensuremath{\mathfrak{g}}}}
\newcommand{\Ggoth}{{\ensuremath{\mathfrak{G}}}}
\newcommand{\pgoth}{{\ensuremath{\mathfrak{p}}}}
\newcommand{\Pgoth}{{\ensuremath{\mathfrak{P}}}}

\newcommand{\lnd}{{\rm LND}}
\newcommand{\Proj}{{\rm Proj\,}}
\newcommand{\Spec}{{\rm Spec\,}}
\renewcommand{\emptyset}{\varnothing}
\newcommand{\ACF}{\text{\rm ACF}_0}
\newcommand{\trdeg}{\text{\rm trdeg}}
\newcommand{\bE}{\mathbf{E}}
\newcommand{\Scho}[1]{\text{\rm{\bf Sch}/{$#1$}}}

\theoremstyle{plain}

\newtheorem{theorem}{Theorem}[section]
\newtheorem{proposition}[theorem]{Proposition}
\newtheorem{lemma}[theorem]{Lemma}
\newtheorem{corollary}[theorem]{Corollary}

\theoremstyle{definition}

\newtheorem{definition}[theorem]{Definition}
\newtheorem{remark}[theorem]{Remark}

\newtheorem{setup}[theorem]{Setup}
\newtheorem{noname}[theorem]{}

\DeclareMathOperator{\SL}{SL}
\DeclareMathOperator{\GL}{GL}
\newcommand{\Aut}{\text{\rm Aut}}
\newcommand{\Frac}{\text{\rm Frac}}
\newcommand{\rank}{\text{\rm rank}}
\newcommand{\setspec}[2]{\big\{\,#1\, \mid \,#2\, \big\}}
\newcommand{\CC}{ \mathcal{C} }

\theoremstyle{definition}

\theoremstyle{remark}

\begin{document}

\makeatletter	   
\makeatother     

\title[the ring of invariants of the binary quintic representation of $\SL_2$]{automorphisms of the ring of invariants \\ of the binary quintic representation of $\SL_2$}
\author{D. Daigle and G. Freudenburg}
\date{\today}
\subjclass[2010]{13A50, 13N15, 14L35, 14R20}
\keywords{$\SL_2$-action, $k$-forms, binary quintic, locally nilpotent derivation, invariant theory} 
\maketitle

\begin{abstract}
Let $k^{[6]}$ denote a polynomial ring in $6$ variables over an algebraically closed field $k$ of characteristic zero
and consider the action of $\SL_2(k)$ on $k^{[6]}$ induced by the irreducible representation of $\SL_2$ of degree $5$ (the binary quintic representation).
We consider the ring $Q = (k^{[6]})^{\SL_2}$ of invariant polynomials and show that $\Aut_k(Q) = k^*$, where $\Aut_k(Q)$ is the group of $k$-algebra automorphisms of $Q$.
Based on this result, we show that the group of $\SL_2$-equivariant polynomial automorphisms of $k^{[6]}$ is isomorphic to $k^*$.
\end{abstract}


\section{introduction}
Let $k$ be an algebraically closed field of characteristic 0.
For each integer $n\ge 1$, let $V_n \subset k[x,y]$ be the vector space (of dimension $n+1$) of homogeneous polynomials of degree $n$
and consider the left-action of $\SL_2(k)$ on $V_n$ given by the following rule:
$$
\text{if $g = \left( \begin{smallmatrix} a & b \\ c & d \end{smallmatrix}\right) \in \SL_2(k)$ and $F(x,y) \in V_n$
then $g * F(x,y) = F(dx-by,-cx+ay)$.}
$$
Define the group homomorphism  $\rho_n : \SL_2(k) \to \GL(V_n)$ by declaring that if $g \in \SL_2(k)$ then $\rho_n(g)$ is the linear map $F(x,y) \mapsto g * F(x,y)$.
Then $\rho_n$ is an irreducible representation of $\SL_2(k)$ on $V_n$, called the irreducible representation of degree $n$.
These are the only irreducible representations of $\SL_2(k)$.  
It is customary to refer to $\rho_5$ as the binary quintic representation of $\SL_2(k)$. 

We may view $V_n$ as an algebraic variety (isomorphic to the affine space $\A^{n+1}_k$),
and the above action as a left-action of $\SL_2$ on the variety $V_n$.
This gives rise to a right-action  of $\SL_2$ on the coordinate algebra\footnote{Given a ring $B$ and $n\in\N$, $B^{[n]}$ denotes the polynomial ring in $n$ variables over $B$.}
$k[V_n] \cong k^{[n+1]}$,
so we may consider the ring of invariants $Q_n = k[V_n]^{ \SL_2 }$.
We may therefore consider the group $\Aut_k(Q_n)$ of $k$-algebra automorphisms of $Q_n$, 
and the group $\Aut_{\SL_2}( V_n )$ of $\SL_2$-equivariant polynomial automorphisms of $V_n \cong \A^{n+1}_k$.

The groups $\Aut_{\SL_2}( V_n )$ were first studied by Kurth in \cite{Kurth.97}, who showed that, when $k=\C$, $\Aut_{\SL_2}( V_n ) \cong \C^*$ for $1\le n\le 4$. 
Note that $Q_1\cong k$, $Q_2\cong Q_3\cong k^{[1]}$ and $Q_4\cong k^{[2]}$, and their automorphism groups are known.
In this paper, we show that $\Aut_k(Q_5) \cong k^*$, and we extend Kurth's result by showing that $\Aut_{\SL_2}( V_5 ) \cong k^*$.

Let $Q = Q_5$.  It is known that $Q$ is a rational unique factorization domain (UFD) of dimension three,
and that $Q$ is generated as a $k$-algebra by elements $x,y,z,w$ where $R:=k[x,y,z]\cong_kk^{[3]}$ and $w^2\in R$. 
The particular relation depends on the choice of generators. For example, Dixmier \cite{Dixmier.90} gives $16w^2=F$ where:
\begin{equation}\label{Cayley}
F=x^2z^3-2xy^2z^2+y^4z-72x^2yz+8xy^3-432x^3 .
 \end{equation}
Dixmier attributes this description of $Q$ to Hermite (1854). In \cite{Salmon.1885}, Article 230, Salmon gives another proof of Hermite's result. 
Replacing $w$ by $4w$ gives the simpler relation $w^2=F$, so
$$
Q\cong_kR[W]/(W^2-F)
$$
where $R[W]=R^{[1]}$ and $F$ is as in \eqref{Cayley}.
 
We obtain ${\rm Aut}_k(Q)\cong k^*$ as a special case of {\it Theorem\,\ref{main3}}. 
Our main results are Theorems \ref{main1}, \ref{main2} and \ref{main3},
and the results about $\Aut_k(Q)$ and $\Aut_{\SL_2}( V_5 )$ are Corollaries \ref{cor1} and \ref{main4}.
These five results are valid over any algebraically closed field $k$ of characteristic zero.

\medskip

Let $k$ be an algebraically closed field of characteristic zero.
Let $R=\bigoplus_{d\in\N}R_d$ be the $\N$-grading of $R=k[x,y,z]\cong k^{[3]}$ where $x\in R_3$, $y\in R_2$ and $z\in R_1$. 
Consider the following conditions on an element $f$ of $R_9$:
\begin{itemize}
\item [{\rm (C1)}] $f$ is irreducible.
\item [{\rm (C2)}] The curve $\Proj (R/fR)$ has at least two singular points.
\item [{\rm (C3)}] $f(x,y,0)=x(px^2+qy^3)$ for $p,q\in k^*$. 
\item [{\rm (C4)}] $f_z(x,y,0)=y(rx^2+sy^3)$ for $r,s\in k$ where $\det\left(\begin{smallmatrix} 3p & r\cr q & s\end{smallmatrix}\right)\ne 0$.
\end{itemize}

\begin{theorem}\label{main1}
If $f\in R_9$ satisfies conditions {\rm (C1)-(C3)} then $|f|_R\ge 2$.
\end{theorem}

The inequality $|f|_R\ge 2$ means that there is no nonzero locally nilpotent derivation $\delta$ of $R$ for which $\delta^2f=0$ (see {\it Section 2}).
The proof uses recent results on rigidity of graded rings, due to the first author and Chitayat \cite{Chitayat.Daigle.22, Daigle.ppt2023}.
(An integral domain $A$ of characteristic zero is said to be {\it rigid\/} if the only locally nilpotent derivation $\delta : A \to A$ is the zero derivation.)
The first of these papers generalizes earlier results of Kishimoto, Prokhorov and Zaidenberg \cite{Kishimoto.Prokhorov.Zaidenberg.13} that relate
the cylindricity of a polarized projective variety $X$ to the existence of nontrivial $\G_a$-actions on the affine cone over $X$.
Whereas the result of \cite{Kishimoto.Prokhorov.Zaidenberg.13} requires the varieties involved to be normal,
the result that we use (part (1) of Theorem 1.2 of \cite{Chitayat.Daigle.22}, or Theorem \ref{edh83yf6r79hvujhxu6wrefji9e} in the present paper) does not.
This is crucial to our argument, since the rings that we consider are not generally normal. 

\begin{theorem}\label{main2}
Let $f\in R_9$ satisfy conditions {\rm (C1)-(C3)} and let $R[W]=R^{[1]}$. 
For every integer $n\ge 2$ with $n\not\in 3\Z$, the quotient ring $R[W]/(W^n+f)$ is a rigid rational UFD. 
\end{theorem}

The proof of this theorem uses {\it Theorem\,\ref{main1}} together with results due to the second author and Moser-Jauslin \cite{Freudenburg.Moser-Jauslin.13}
which characterize the homogeneous locally nilpotent derivations of $R[W](W^n-r)$, $r\in R$, in terms of the homogeneous locally nilpotent derivations of $R$. 

\begin{theorem}\label{main3}
Let $f\in R_9$ satisfy conditions {\rm (C1)-(C4)} and let $A = R[W]/(W^n+f)$ where $R[W]=R^{[1]}$ and $n\ge 2$ is such that $n\not\in 3\Z$.
Let $\Ggoth$ be the $\N$-grading of $A$ defined by declaring that $x \in A_{3n}$,  $y \in A_{2n}$,  $z \in A_{n}$ and $w \in A_{9}$,
and let $T(\Ggoth)$ be the subtorus of $\Aut_k(A)$ determined by $\Ggoth$.
Then $\Aut_k(A) = T(\Ggoth) \cong k^*$.
\end{theorem} 

See paragraph \ref{7t65fr56gb12dwe8} for details about $T(\Ggoth)$.
This description of the automorphism group relies on {\it Theorem\,\ref{main2}} and the theorem of Arzhantsev and Gaifullin \cite{Arzhantsev.Gaifullin.17} 
which asserts that the group of algebraic automorphisms of a rigid affine $k$-variety contains a unique maximal torus
(uniqueness of the maximal torus was first proved for rigid surfaces by Flenner and Zaidenberg \cite{Flenner.Zaidenberg.05b}). 

Since the polynomial $F$ in line (\ref{Cayley}) satisfies conditions (C1)-(C4), as the reader can check, Theorem \ref{main3} has the following special case:

\begin{corollary} \label {cor1}
${\rm Aut}_k(Q)\cong k^*$
\end{corollary}

We also have:

\begin{corollary}\label{main4}
$\Aut_{\SL_2}( V_5 ) \cong k^*$
\end{corollary}

We will see that the case $k=\C$ of {\it Corollary\,\ref{main4}} follows from {\it Corollary\,\ref{cor1}} and a result of \cite{Kurth.97},
and we will show that the hypothesis $k=\C$ can be removed.

\medskip

\paragraph{\bf Acknowledgment.} The authors wish to thank Frank Kutzschebauch of the University of Bern for bringing the paper of Kurth to their attention.

\section{preliminaries}   \label {SEC:preliminaries}

For a field $k$,  a {\bf $k$-domain} is an integral domain which is also a $k$-algebra.
Given a ring $B$ and $n\in\N$, $B^{[n]}$ denotes the polynomial ring in $n$ variables over $B$. 
If $B$ is a ring and $b \in B$, we write $B_b$ for the localized ring $S^{-1}B$ with $S = \{1,b,b^2,b^3,\dots\}$.

\paragraph{\bf Locally nilpotent derivations.}\label{LNDS} 
We gather several definitions and properties for locally nilpotent derivations used in subsequent parts of the paper.
The reader is referred to \cite{Freudenburg.17} for a detailed treatment of the subject. 

Assume that $k$ is a field of characteristic $0$ and that $B$ is a $k$-domain.

A {\bf locally nilpotent derivation} of $B$ is a $k$-derivation $D : B \to B$ such that, for each $f\in B$, there exists $n\in\N$ with $D^nf=0$.
The set of locally nilpotent derivations of $B$ is denoted ${\rm LND}(B)$. 
One says that $B$ is {\bf rigid} if ${\rm LND}(B)=\{ 0\}$. 

Let $D \in \lnd(B)$.
Given $f \in B \setminus \{0\}$, the number $\deg_D(f) = -1 + \min\setspec{ n \in \N }{ D^n f =0 }$ is called the {\bf $D$-degree} of $f$;
one also defines $\deg_D(0) = -\infty$.
The map $\deg_D : B \to \N \cup \{-\infty\}$ satisfies the following conditions for all $f,g \in B$:
(i)~$\deg_D(f) < 0$ $\Leftrightarrow$ $f=0$;
(ii)~$\deg_D(fg) = \deg_D(f) + \deg_D(g)$;
(iii)~$\deg_D(f+g) \le \max( \deg_D(f) , \deg_D(g) )$.

Given nonzero $f\in B$ the {\bf absolute degree} of $f$, denoted $|f|_B$, is the minimum value of $\deg_D(f)$ for nonzero $D\in {\rm LND}(B)$.
If there are no nonzero $D\in {\rm LND}(B)$ (i.e., if $B$ is rigid) then we define $|f|_B = +\infty$. 

One says that a derivation $D : B \to B$ is {\bf irreducible} if $B$ is the only principal ideal of $B$ that contains $D(B)$.

Let $D\in {\rm LND}(B)$ and $A=\krn D = \setspec{ f \in B }{ Df=0 }$.
Since $A = \setspec{ f \in B }{ \deg_D(f) \le 0}$, it follows that $A$ is factorially closed in $B$, i.e.,
if $a,b \in B \setminus \{0\}$ satisfy $ab\in A$ then $a\in A$ and $b\in A$. 
The set $A \cap D(B)$ is an ideal of $A$, called the {\bf plinth ideal} of $D$.
If $f \in B$ satisfies $\deg_D(f)=1$ (equivalently, $Df \neq 0$ and $D^2f=0$), one says that $f$ is a {\bf local slice} for $D$;
note that if $D \neq 0$ then $D$ has a local slice.
A {\bf slice} for $D$ is an element $s \in B$ satisfying $Ds=1$ (so every slice is a local slice).
A local slice $f$ is said to be {\bf minimal} if, whenever $A[f]\subseteq A[g]$ for another local slice $g$, we have $A[f]=A[g]$. 

If $B=\bigoplus_{i\in\Z}B_i$ is a $\Z$-grading then $D$ is {\bf homogeneous} if there exists $d\in\Z$ such that $DB_i\subseteq B_{i+d}$ for all $i\in\Z$.
If $D \neq 0$ then this $d$ is unique and is called the {\bf degree} of $D$.

For the polynomial ring $B=k^{[n]}$ and $D\in {\rm LND}(B)$, the {\bf corank} of $D$ equals 
the greatest integer $r$ such that there exists a system of variables $B=k[x_1,\hdots ,x_n]$ with $Dx_1=\cdots =Dx_r=0$. The {\bf rank} of $D$ is $n-r$. 

\smallskip

Throughout \ref{slice}--\ref{rank-two}, $k$ is an arbitrary field of characteristic zero.

\begin{lemma}\label{slice} {\rm (See \cite{Freudenburg.17}, Principle 11)}
Given $D\in {\rm LND}(B)$ let $A=\krn D$. If $f\in B$ is a local slice for $D$ and $a=Df$, then $B_a=A_a[f]\cong A_a^{[1]}$. 
\end{lemma} 

\begin{proposition}\label{AB-theorem}  {\rm (\cite{Freudenburg.17}, Theorem 2.50)} Let $B$ be a $k$-domain and $D\in {\rm LND}(B)$, $D\ne 0$. 
Suppose that $u,v\in\krn D$ and $x,y\in B$ are nonzero, and $a,b\in\Z$, $a,b\ge 2$, are such that $ux^a+vy^b\ne 0$. 
\begin{itemize}
\item [{\bf (a)}]  If $D(ux^a+vy^b)=0$ then $Dx=Dy=0$.
\item [{\bf (b)}]  If $D^2(ux^a+vy^b)=0$ and $(a,b) \neq (2,2)$ then $Dx=Dy=0$. 
\end{itemize}
\end{proposition} 

\begin{proposition}\label{FMJ13} 
Let $R$ be a $\Z$-graded affine $k$-domain.
Suppose that $f\in R$ is homogeneous, $\deg f\ne 0$ and $n\ge 2$ is an integer relatively prime to $\deg f$. 
Let $B=R[W]/(W^n+f)$ where $R[W]=R^{[1]}$. 
\begin{itemize}
\item [{\bf (a)}] $B$ is a $k$-domain and ${\rm frac}(B)\cong {\rm frac}(R)$. 
\item [{\bf (b)}] If $R$ is a UFD and $f$ is prime in $R$, then $B$ is a UFD. 
\item [{\bf (c)}]  $B$ is rigid if and only if $|f|_R\ge 2$.
\end{itemize}
\end{proposition}

Parts (a) and (b) of this result follow from \cite{Daigle.Freudenburg.Nagamine.22}, Theorem 3.8. Part (c) follows from \cite{Freudenburg.Moser-Jauslin.13}, Corollary 3.2. 

\paragraph{\bf The polynomial ring $k^{[3]}$.}

We give some results about locally nilpotent derivations of the polynomial ring $k^{[3]}$.

\begin{theorem}\label{Miyanishi}
Let $B=k^{[3]}$. If $D\in {\rm LND}(B)$ and $D\ne 0$ then $\krn D \cong k^{[2]}$. 
\end{theorem} 

Miyanishi \cite{Miyanishi.85} showed this result in the case $k=\C$; the general statement is deduced using a theorem of Kambayashi; see \cite{Daigle.Kaliman.09}. 

\begin{proposition}\label{plinth} 
Let $B=k^{[3]}$, $D\in {\rm LND}(B) \setminus \{0\}$ and $A=\krn D$. 
\begin{itemize}
\item [{\bf (a)}] The plinth ideal $A\cap DB$ for $D$ is a principal ideal of $A$.
\item [{\bf (b)}] A local slice $f$ for $D$ is minimal if and only if $Df$ generates the plinth ideal of $D$. 
\item [{\bf (c)}] There exists a minimal local slice of $D$.
\item [{\bf (d)}] If $r$ is a minimal local slice of $D$ then $\setspec{ \alpha r + \beta }{ \alpha,\beta \in A \text{ and } \alpha \neq 0}$ is the set of all local slices of $D$.
\end{itemize}
\end{proposition}

\begin{proof}
Part (a) is Theorem 1 of \cite{Daigle.Kaliman.09}, part (b) is an exercise, and parts (c) and (d) follow from (a) and (b).
\end{proof}

\begin{proposition}\label{jacobian} {\rm (See \cite{Daigle.97}.)} Let $B=k^{[3]}$, let $D\in {\rm LND}(B)$ be irreducible, and let $f,g\in B$ be such that $\krn D=k[f,g]$. Up to multiplication by an element of $k^*$, $D$ is given by the jacobian determinant
\[
D = \left| \frac{\partial(f,g, \, \underline{\ \ } \, )}{\partial(x,y,z)} \right| \, .
\]
\end{proposition}

\begin{proposition}\label {Daigle00result19}
Let $a,b,c$ be positive integers
and let $B = k[x,y,z] = k^{[3]}$ be endowed with the $\N$-grading in which $x,y,z$ are homogeneous of degrees $a,b,c$.
Let $D\in {\rm LND}(B)$ be irreducible and homogeneous. 
\begin{itemize}
\item [{\bf (a)}] There exist homogeneous $F,G\in B$ with $\krn D=k[F,G]$. 
\item [{\bf (b)}] If $d=\deg F$ and $e=\deg G$ then $\deg (D)=(d+e)-(a+b+c)$. 
\item [{\bf (c)}] If $a,b,c$ are pairwise relatively prime then $\gcd (d,e)=1$. 
\end{itemize}
\end{proposition}

\begin{proof} Part (a).\footnote{This result was first shown by Zurkowski, independent of Miyanishi's theorem, in the preprint \cite{Zurkowski.ppt2}. 
Since this paper was not published and may not be widely available, we include a short proof based on Miyanishi's theorem.}
By {\it Theorem\,\ref{Miyanishi}} there exist $f,g\in B$ such that $\krn D=k[f,g]$. 
Let $f=\sum_{i=m}^nf_i$ and $g=\sum_{j=p}^qg_j$ be their respective decompositions into homogeneous summands. Since $D$ is homogeneous, $Df_i=0$ and $Dg_j=0$ for each $i,j$, and we see that 
$k[f,g]=k[f_m,\hdots ,f_n,g_p,\hdots ,g_q]$. By \cite{Daigle.00}, 1.6 (or see \cite{Freudenburg.17}, Proposition 3.42 for details),
there exist $F,G\in\{ f_m,\hdots ,f_n,g_p,\hdots ,g_q\}$ such that $\krn D=k[F,G]$. 

Part (b). This follows from the jacobian formulation for $D$ given in {\it Proposition\,\ref{jacobian}}. 

Part (c). This is \cite{Daigle.00}, result 1.9. 
\end{proof}

\begin{lemma} \label {723476ed7dh0927gb}
Let $a,b,c$ be pairwise relatively prime positive integers and
let $R = k[x,y,z] = k^{[3]}$ be endowed with the $\N$-grading such that $x,y,z$ are homogeneous of degrees $a,b,c$.
Let $s$ be a nonzero homogeneous element of $R$ satisfying $\delta s\ne 0$ for all nonzero $\delta\in {\rm LND}(R)$. 
Assume that $D\in\lnd(R)$ is irreducible and homogeneous and $D^2(s)=0$. Let $\krn D=k[u,v]$ for homogeneous $u,v$. 
\begin{itemize}
\item [{\bf (a)}]  $Ds\notin k[u]$ and $Ds\notin k[v]$. 
\item [{\bf  (b)}]  If $\deg (s)\ne a+b+c$ then $\min (\deg (u),\deg (v)) \le \max\{ 2, \deg(s)-(a+b+c)\}$.
\end{itemize}
\end{lemma}

\begin{proof}
Part (a). Assume that $Ds \in k[u]$; then $Ds = \lambda u^i$ for some $\lambda \in k^*$ and $i\ge0$.
{\it Lemma\,\ref{slice}} implies $R_u = A_u[s] = k[u^{\pm1}, s][v]$.
Therefore, for each $j\ge0$, $u^j \frac{\partial}{\partial v} : R_u \to R_u$  is a nonzero locally nilpotent derivation with kernel $k[u^{\pm1},s]$.
For a suitable choice of $j$, $u^j \frac{\partial}{\partial v}$ maps $R$ into itself, so the restriction $D' : R \to R$ of $u^j \frac{\partial}{\partial v}$
is a nonzero locally nilpotent derivation of $R$ such that $D's=0$, contradicting the hypothesis that no such $D'$ exists.
Therefore, $Ds\notin k[u]$ and the same argument shows $Ds\notin k[v]$.

Part (b). Define
\[
m = \deg(u)\,\, ,\,\, n = \deg(v)\,\, ,\,\, \omega = \deg(s) - (a +b +c)
\] 
noting that $\omega\ne 0$ by hypothesis. {\it Proposition\,\ref{Daigle00result19}} implies $\deg(D) = (m+n)-(a+b+c)$ and, since
$a,b,c$ are pairwise relatively prime positive integers,  $\gcd(m,n)=1$.
Since $Ds \neq 0$,
it follows that $\deg( Ds ) = m+n+\omega$, i.e., $Ds$ is a nonzero homogeneous element of $k[u,v]$ of degree $m+n+\omega$.
Note that if $u^iv^j$ is a monomial appearing in $Ds$ then $\deg (Ds)=mi+nj$. 
Therefore, the set 
\[
S := \{ (i,j) \in \N^2 \, |\, mi+nj = m+n+\omega \}
\]
is not empty.

Proceeding by contradiction, assume that $\max(2, \omega) < \min(m,n)$. 
Since $\omega\ne 0$ we see that $\omega\notin m\N+n\N$. For any $(i,j)\in S$ we have $\omega =(i-1)m+(j-1)n$, so:
\begin{equation} \label {ieq0orjeq0}
\text{If $(i,j) \in S$ then $i=0$ or $j=0$.}
\end{equation}
We claim that $|S|=1$. If this is not the case, then \eqref{ieq0orjeq0} implies that there exist $i,j>0$ such that $(i,0), (0,j) \in S$.
Then $mi = m+n+\omega = nj$. Since $\gcd(m,n)=1$, there exists an integer $\sigma >0$ such that $j=\sigma m$ and $i=\sigma n$.
It follows that $m+n+\omega = \sigma mn$ and we have:
\[
\sigma mn-m-n = \omega \notin m\N+n\N 
\]
Recall that the Frobenius number of the numerical semigroup $m\N+n\N$ is the largest integer not in $m\N+n\N$ and that this number is given by $mn-m-n$.
We thus see that $\sigma =1$ and $\omega = mn-m-n$.
Using $\min(m,n)>2$, we easily obtain $mn-m-n>\min(m,n)$, which contradicts the fact that $\max(2,\omega) < \min(m,n)$.
This contradiction shows that $|S|=1$.

From $|S| = 1$, it follows that $Ds = \lambda u^i v^j$ for some $\lambda \in k^*$ and $i,j\in\N$. But then $i=0$ or $j=0$ by \eqref{ieq0orjeq0},
giving a contradiction to part (a). 

Therefore, $\min (m,n)\le \max (2,\omega )$.
\end{proof}

\begin{lemma}\label {rank-two}
Let $B=k[x,y,z]=k^{[3]}$. Suppose that $D\in {\rm LND}(B)$ is irreducible and $Dx=0$. 
There exist $p\in B$ and $f(x)\in k[x] \setminus \{0\}$ such that $\krn D=k[x,p]$ and the plinth ideal of $D$ is $f(x)k[x,p]$. 
\end{lemma}

\begin{proof} The existence of $p$ such that $A:=\krn D=k[x,p]$ is given by \cite{Daigle.Freudenburg.98}, Corollary 3.2 part (3). Part (5) of this same result shows that $k[x]\cap DB\ne\{ 0\}$.
So $k[x]\cap DB$ is a nonzero ideal of $k[x]$. Choose $f(x)\in k[x]$ such that $k[x]\cap DB=f(x)k[x]$. By {\it Proposition\,\ref{plinth}}, there exists $a\in A$ such that $A\cap DB=aA$. 
In particular, $f\in aA$.  Clearly, $k[x]$ is factorially closed in $B$. 
Therefore, $a\in k[x]$, and $a\in k[x]\cap DB$ gives $a\in fA$. It follows that $aA=fA$ and $A\cap DB=fA$. 
\end{proof}


\paragraph{\bf Automorphism Groups.}

\begin{noname} \label {7t65fr56gb12dwe8}
Let $k$ be an infinite field and $A$ a $k$-domain.

If $H$ is a subgroup of $\Aut_k(A)$, we write $\CC(H)$ for the centralizer of $H$ in $\Aut_k(A)$.

Let $\ggoth$ be a $\Z^n$-grading of $A$,  $A = \bigoplus_{i \in \Z^n} A_i$.
For each $t \in (k^*)^n$, define $\phi_t \in \Aut_k(A)$ by $\phi_t ( \sum_{i \in \Z^n} a_i ) = \sum_{i \in \Z^n} t^i a_i$
(where $a_i \in A_i$ for all $i$, and where $t^i a_i$ is an abbreviation for the product $t_1^{i_1} \cdots t_n^{i_n} a_{i_1,\dots,i_n}$).
We shall use the notation
$$
T(\ggoth) = \setspec{ \phi_t }{ t \in (k^*)^n } .
$$
It is clear that $T(\ggoth)$ is a subgroup of $\Aut_k(A)$, and the following property of $T(\ggoth)$ is well known:
\begin{equation} \label {mnbkjhgertypoi345098asdb7h2wsgt}
\CC( T(\ggoth) ) = \setspec{ \alpha \in \Aut_k(A) }{ \alpha(A_i) = A_i \text{ for all $i \in \Z^n$} } .
\end{equation}
Moreover, one can see that
\begin{equation} \label {23ws5rf8uyguytrpoimnlk098uzxasdf3}
\begin{minipage}{.8\textwidth}
if $k$ is algebraically closed then $T(\ggoth)$ is isomorphic to $(k^*)^r$,
where $r$ is the rank of the subgroup of $\Z^n$ generated by $\setspec{ i \in \Z^n }{ A_i \neq 0 }$.
\end{minipage}
\end{equation}
\end{noname}

\begin{definition} \label {lkjmnbv987y123erfpoiusdfv8gut}
Let $k$ be an infinite field and $A$ a $k$-domain.
A subgroup $T$ of $\Aut_k(A)$ is called a {\it subtorus of $\Aut_k(A)$} if there exist an $n \in \N$ and a $\Z^n$-grading $\ggoth$ of $A$ such that $T=T(\ggoth)$.
\end{definition}

\begin{remark}
Let $k$ be an infinite field and $A$ a $k$-domain.
If $T$ is a subtorus of $\Aut_k(A)$ and $\alpha \in \Aut_k(A)$ then $\alpha T \alpha^{-1}$ is a subtorus of $\Aut_k(A)$.
Indeed, we have $T = T(\ggoth)$ for some grading $\ggoth = \big( A = \bigoplus_{i \in \Z^n} A_i \big)$.
Then $\alpha T \alpha^{-1} = T(\ggoth')$ where $\ggoth' = \big( A = \bigoplus_{i \in \Z^n} \alpha(A_i) \big)$.
\end{remark}

\begin{theorem}\label{Arz-Gai} {\rm \cite{Arzhantsev.Gaifullin.17}}
Let $k$ be an algebraically closed field of characteristic zero and $B$ be a rigid affine $k$-domain.
There is a subtorus $\mathbb{T}$ of ${\rm Aut}_k(B)$ that contains any other subtorus of ${\rm Aut}_k(B)$.
Moreover, $\mathbb{T}$ and $\mathcal{C}(\mathbb{T})$ are normal subgroups of $\Aut_k(B)$.
\end{theorem}

\begin{proposition}  \label {CpnoiuB387vbxncAru}
Let $k$ be an infinite field, $A$ a $k$-domain and $\ggoth$ a $\Z^n$-grading of $A$ given by $A = \bigoplus_{i \in \Z^n} A_i$, 
and consider the submonoid $S = \setspec{ i \in \Z^n }{ A_i \neq 0 }$ of $\Z^n$.
If $T(\ggoth)$ is a normal subgroup of $\Aut_k(A)$ then
there exists a group homomorphism $\Aut_k(A) \to \Aut(S)$, $\alpha \mapsto \tilde\alpha$, such that:
\begin{equation}  \tag{$*$}
\text{for each $\alpha \in \Aut_k(A)$ and $i \in S$, \ \   $\alpha(A_i) = A_{ \tilde\alpha(i) }$.}
\end{equation}
Moreover, the kernel of the homomorphism $\alpha \mapsto \tilde\alpha$ is $\CC(T(\ggoth))$.
\end{proposition}

\begin{proof}
Let $T = T(\ggoth)$.
We have $T = \setspec{ \phi_t }{ t \in (k^*)^n }$ where $\phi_t$ is defined in paragraph \ref{7t65fr56gb12dwe8}.
For each $i \in \Z^n$, let $p_i : A \to A_i$ be the canonical projection.
Let $\alpha \in \Aut_k(A)$. We claim:
\begin{equation}  \label {i7u65e432z34f61qj9x1mw}
\text{For each $i \in \Z^n$, there exists $j \in \Z^n$ such that $\alpha(A_i) \subseteq A_j$.}
\end{equation}
Indeed, assume that \eqref{i7u65e432z34f61qj9x1mw} is false.
Then there exist $i,j_1, j_2 \in \Z^n$ such that $j_1 \neq j_2$, $p_{j_1}\big( \alpha(A_i) \big) \neq 0$ and $p_{j_2}\big( \alpha(A_i) \big) \neq 0$.
Pick $x_1 \in A_i$ such that $p_{j_1}( \alpha(x_1) ) \neq 0$
and $x_2 \in A_i$ such that $p_{j_2}( \alpha(x_2) ) \neq 0$;
then some $x \in \{x_1,x_2,x_1+x_2\}$ satisfies  $p_{j_1}( \alpha(x) ) \neq 0$ and $p_{j_2}( \alpha(x) ) \neq 0$.
Write $\alpha(x) = \sum_{j \in \Z^n} y_j$ where $y_j \in A_j$ for all $j$. Note that $y_{j_1} \neq 0$ and $y_{j_2} \neq 0$.
Since $k$ is an infinite field, we can choose $s \in (k^*)^n$ such that $s^{j_1} \neq s^{j_2}$.
Since $T$ is normal in $\Aut_k(A)$, we have $\alpha^{-1} \circ \phi_s \circ \alpha = \phi_t$ for some $t \in (k^*)^n$.
The equalities
$$
\textstyle
\sum_{j \in \Z^n} s^j y_j = \phi_s( \sum_{j \in \Z^n} y_j ) = \phi_s(\alpha(x)) = \alpha(\phi_t(x)) = \alpha(t^i x) = t^i \alpha(x) = \sum_{j \in \Z} t^i y_j 
$$
show that $s^j y_j = t^i y_j$ for all $j \in \Z^n$.
Since  $y_{j_1} \neq 0$ and $y_{j_2} \neq 0$, we get $s^{j_1} = t^i = s^{j_2}$, which contradicts our choice of $s$.
This proves \eqref{i7u65e432z34f61qj9x1mw}.

Statement \eqref{i7u65e432z34f61qj9x1mw} implies that there exists a set map $\tilde{\alpha} : S \to S$ such that
$\alpha(A_i) \subseteq A_{\tilde{\alpha}(i)}$ for all $i \in S$.
Note that $\tilde\alpha$ is uniquely determined by $\alpha$. If $i,j \in S$ then
$$
A_{ \tilde{\alpha}(i+j) } \supseteq \alpha( A_{i+j} ) \supseteq \alpha( A_i A_j ) = \alpha(A_i) \alpha(A_j) \subseteq A_{\tilde{\alpha}(i)} A_{\tilde{\alpha}(j)} \subseteq A_{\tilde{\alpha}(i) + \tilde{\alpha}(j)} \, ,
$$
so $A_{ \tilde{\alpha}(i+j) } \cap  A_{\tilde{\alpha}(i) + \tilde{\alpha}(j)} \neq 0$ and hence $\tilde{\alpha}(i+j) = \tilde{\alpha}(i) + \tilde{\alpha}(j)$.
Thus, $\tilde\alpha$ is an endomorphism of the monoid $S$.
It is easy to check that $\widetilde{\text{\rm id}_A} = \text{\rm id}_S$
and that if $\alpha,\beta \in \Aut_k(A)$ then $\widetilde{\alpha \circ \beta} = \tilde{\alpha} \circ \tilde{\beta}$;
so $\alpha \mapsto \tilde\alpha$ is a homomorphism of groups from $\Aut_k(A)$ to $\Aut(S)$.
The fact that $\widetilde{ \alpha^{-1} } = {\tilde\alpha}^{-1}$ implies that the inclusion $\alpha(A_i) \subseteq A_{\tilde\alpha(i)}$ is in fact an equality.
This shows that there exists a group homomorphism $\Aut_k(A) \to \Aut(S)$, $\alpha \mapsto \tilde\alpha$, satisfying $(*)$.
By \eqref{mnbkjhgertypoi345098asdb7h2wsgt}, the kernel of this homomorphism is $\CC(T(\ggoth))$.
\end{proof}

\begin{corollary} \label {E2sd3sdfpoiu213er098asdCmnasdlkj9w87y}
Let $k$ be an infinite field, $A$ a $k$-domain and $\ggoth$ an $\N$-grading of $A$.
If $T(\ggoth)$ is a normal subgroup of $\Aut_k(A)$ then $\Aut_k(A) = \CC(T(\ggoth))$.
\end{corollary}

\begin{proof} 
Let $A = \bigoplus_{i \in \N} A_i$ be the grading $\ggoth$ and consider the monoid $S = \setspec{ i \in \N }{ A_i \neq 0 }$.
The only automorphism of $S$ is the identity map of $S$.
So $\Aut_k(A)$ is equal to the kernel of the group homomorphism $\Aut_k(A) \to \Aut(S)$, $\alpha \mapsto \tilde\alpha$, of {\it Proposition\,\ref{CpnoiuB387vbxncAru}}. 
That proposition states that the kernel of this homomorphism is equal to $\CC(T(\ggoth))$, so $\Aut_k(A) = \CC(T(\ggoth))$.
\end{proof}

\begin{corollary}  \label {5cFt2xwdlk12pwo09284x2ednBb7w36gt1}
Let $k$ be an algebraically closed field of characteristic zero,
$B$ a rigid affine $k$-domain and $\ggoth$ an $\N$-grading of $B$.
If $\CC( T(\ggoth) ) = T(\ggoth)$ then $\Aut_k(B) = T(\ggoth)$.
\end{corollary}

\begin{proof} 
Let $T = T(\ggoth)$.
By {\it Theorem\,\ref{Arz-Gai}}, there exists a subtorus $\field{T}$ of $\Aut_k(B)$ that contains any other subtorus of ${\rm Aut}_k(B)$.
Thus, $T \subseteq \field{T}$. Since $\field{T}$ is abelian, $\field{T} \subseteq \CC(T) = T$, so $\field{T} = T$.
{\it Theorem\,\ref{Arz-Gai}} also states that $\field{T} = T$ is normal in $\Aut_k(B)$, so {\it Corollary\,\ref{E2sd3sdfpoiu213er098asdCmnasdlkj9w87y}}
implies that $\Aut_k(B) = \CC(T) = T$.
\end{proof}

\paragraph{\bf Further preliminary results.}
For the remainder of {\it Section \ref{SEC:preliminaries}}, we assume that $k$ is an algebraically closed field of characteristic zero.

We say that a ring $R$ is a polynomial ring in one variable over a subring if there exists a subring $A\subset R$ and $t\in R$ such that $R=A[t]\cong_AA^{[1]}$.
Given  a $\Z$-graded integral domain $B$ and a nonzero homogeneous $\xi\in B$, $B_{\xi}$ denotes the localization of $B$ at $\{ 1,\xi ,\xi^2,\hdots\}$ and 
$B_{(\xi )}$ denotes the degree-0 subring of the graded ring $B_{\xi}$.

\begin{definition} \label {6gy2cmiwv6fta2w3dx1hj2iew}
Let $B$ be a $\Z$-graded integral domain. An element $\xi\in B$ is {\bf cylindrical} if it is nonzero, homogeneous, $\deg \xi\ne 0$
and $B_{(\xi )}$ is a polynomial ring in one variable over a subring. 
\end{definition}

\begin{theorem}[\cite{Chitayat.Daigle.22}, Part (1) of Theorem 1.2] \label {edh83yf6r79hvujhxu6wrefji9e} 
Let $B = \bigoplus_{i \in \Z}B_i$ be a $\Z$-graded affine domain over a field of characteristic zero.
Assume that one of the following conditions {\rm(i)}, {\rm(ii)} is satisfied:
\begin{itemize}

\item[\rm(i)] the transcendence degree of $B$ over $B_0$ is at least $2$;

\item[\rm(ii)] there exist $i,j$ such that $i < 0 < j$, $B_i \neq 0$ and $B_j \neq 0$.

\end{itemize}
Then the following are equivalent:
\begin{itemize}

\item[{\rm (a)}] There exists $d \ge 1$ such that $B^{(d)}$ is non-rigid,
where $B^{(d)} = \bigoplus_{ i \in \Z } B_{id}$.

\item[{\rm (b)}] $B$ has a cylindrical element.

\end{itemize}
\end{theorem}

\begin{proposition}  \label {CriterionKernel} 
Let $R$ be a $3$-dimensional affine $k$-domain and let $f$ be a prime element of $R$.
Then {\rm(a)}$\Rightarrow${\rm(b)}$\Rightarrow${\rm(c)}, where:
\begin{itemize}
\item [{\rm (a)}] There exists a nonzero $\delta \in \lnd(R)$ such that $\delta (f)=0$.
\item [{\rm (b)}] The ring $R/fR$ is not rigid.
\item [{\rm (c)}] For each $\N$-grading $R = \bigoplus_{d \in \N} R_d$ such that $f$ is homogeneous and $R_0 = k$, $\Proj(R/fR )$ is a unicuspidal
rational curve.\footnote{A {\it unicuspidal rational curve\/} is a projective curve $C$ for which there exists a point $P \in C$ such that $C \setminus \{P\}$
is isomorphic to $\A^1_k$. Equivalently, it is a projective curve that contains an open set isomorphic to $\A^1_k$.}
\end{itemize}
\end{proposition}

\begin{proof}
If (a) holds then we can choose $\delta$ to be irreducible; then $\delta$  induces a nonzero locally nilpotent derivation of $R/fR$, showing that (b) holds.
If (b) holds, and if we are given a grading satisfying the hypothesis of (c),
then {\it Theorem\,\ref{edh83yf6r79hvujhxu6wrefji9e}} implies that the graded ring $R/fR$ has a cylindrical element.
It follows that $\Proj(R/fR)$ contains an open set isomorphic to $\A_k^1$, from which (c) follows.
\end{proof}

\begin{proposition}  \label {8y23874t28746rdhjh} 
Let $R=k[x,y,z]=k^{[3]}$ with $\N$-grading $R=\bigoplus_{i\in \N}R_i$ where $x,y,z$ are homogeneous of positive degrees.
Let $f\in R$ be a homogeneous prime element such that $\Proj(R/fR )$ is not a unicuspidal rational curve.
\begin{itemize}
\item [{\bf (a)}] The quotient ring $R/(f-c)$ is rigid for every $c\in k$.
\item [{\bf (b)}] Given $\delta \in \lnd(R)$, if $\delta\ne 0$ then $\delta (f)\ne 0$.
\end{itemize}
\end{proposition}

\begin{proof}
Since $\Proj(R/fR)$  is not a unicuspidal rational curve, {\it Proposition\, \ref{CriterionKernel}} implies that $R/fR$ is a rigid domain and that (b) holds. 
Since $R/fR$ is rigid, Proposition 10.5 of \cite{Daigle.ppt2023} implies that $R/(f-c)$ is rigid for every $c \in k$, so (a) holds.
\end{proof}

For the next two results, 
let $R=\bigoplus_{d\in\N}R_d$ be the $\N$-grading of $R=k[x,y,z]\cong k^{[3]}$ where $x\in R_3$, $y\in R_2$ and $z\in R_1$.

\begin{proposition}\label{curve-singularities} 
If $f\in R_9$ is irreducible and $\deg_z(f)=1$ then the curve $C={\rm Proj}(R/fR)$ has exactly one singular point. 
\end{proposition} 

\begin{proof}
Write $f=\alpha z+\beta$ for homogeneous $\alpha ,\beta\in k[x,y]$.
We see that $\deg\alpha =8$, so $\alpha$ is in the $k$-span of $x^2y$ and $y^4$, and that 
$\deg \beta =9$, so $\beta$ is in the $k$-span of $x^3$ and $xy^3$. Let $p,q,r,s\in k$ be such that $f=y(px^2+qy^3)z+x(rx^2+sy^3)$.
Since $f$ is irreducible we see that $q\ne 0$, $r\ne 0$ and $ps\ne qr$. 
Up to an automorphism of $R$ (as a graded $k$-algebra) and multiplication by an element of $k^*$, we may arrange that 
$$
f=y(ax^2+y^3)z+x(x^2+by^3) \qquad \text{for some $a,b \in k$ such that $a b \neq 1$.}
$$
We view $C$ as a subvariety of the weighted projective plane $\PP (3,2,1) = \Proj(R)$.
Note that $C \subseteq D_+(z) \cup D_+(y)$.

We have $D_+(z) = \Spec R_{(z)}$ and $R_{(z)} = k[u,v] = k^{[2]}$ where $u = x/z^3$ and $v = y/z^2$.
So $D_+(z)$ is an affine plane and the equation of $C \cap D_+(z)$ is
$v(au^2+v^3)+u(u^2+bv^3) = 0$, which we rewrite as
$$
(au^2v + u^3) + (v^4+buv^3) = 0 .
$$
This is a curve of degree $4$ with a singularity of multiplicity $3$ at the origin $O$ of this affine plane.
It follows that $C$ is a rational curve and that $O$ is the only singular point of $C$ in $D_+(z)$.
To show that $O$ is the only singular point of $C$, it suffices to show that each point of $C \setminus D_+(z)$ is a regular point of $C$.
We have $C \setminus D_+(z) = \{ [0:1:0] , [\sqrt{-b}:1:0] \}$.
(Note that $[{\sqrt {-b}}:1:0]$ and $[-{\sqrt {-b}}:1:0]$ are the same point of $\PP(3,2,1)$.)

We have $D_+(y) = \Spec R_{(y)}$ and $R_{(y)} = k[u,v,w]$ with $u = x^2/y^3$, $v = z^2/y$ and $w = xz/y^2$.
Since $f$ is prime in $R$ it is also prime in $R_y$, so $\pgoth = R_{(y)} \cap f R_y$ is a prime ideal of $R_{(y)}$.
Moreover,  $C \cap D_+(y)$ is the subset $V(\pgoth)$ of $\Spec( R_{(y)} )$.
One can see that $\pgoth$ is generated by the following two elements of $R_{(y)}$:
$$
\textstyle   \frac{xf}{y^6} = a u w + u^2 + bu + w , \quad \frac{zf}{y^5} = a u v + u w + v + b w .
$$
Since $R_{(y)} \cong k[U,V,W] / (UV-W^2)$ (where $k[U,V,W] = k^{[3]}$), we have 
$R_{(y)} / \pgoth \cong k[U,V,W] / \Pgoth$, where
$$
\Pgoth = ( a U W + U^2 + bU + W , a U V + U W + V + b W , \, UV-W^2) \subset k[U,V,W] .
$$
So the curve  $C \cap D_+(y)$ is isomorphic to $V(\Pgoth) \subset \mathbb{A}_k^3$.
The jacobian matrix of the three generators of $\Pgoth$ is the matrix $J$ displayed here:
\begin{equation}
J = \left(\begin{smallmatrix}
a W + 2 U + b  &  0  &  a U + 1   \\
a V + W  &  a U + 1 &  U + b  \\
V & U & -2W
\end{smallmatrix}\right) , \qquad
J_0 = \left(\begin{smallmatrix}
b  &  0  &  1   \\
0  &  1 &  b  \\
0 & 0 & 0
\end{smallmatrix}\right) , \qquad
J_1 = \left(\begin{smallmatrix}
-b  &  0  &  1-ab \\
0  & 1-ab &  0  \\
0 & -b & 0
\end{smallmatrix}\right) .
\end{equation}
The point $[0:1:0] \in D_+(y) \subset \PP(3,2,1)$ corresponds to the point $(U,V,W)=(0,0,0) \in \mathbb{A}_k^3$,
and $J$ evaluated at that point is $J_0$, which has rank $2$; so $[0:1:0]$ is a regular point of $C$.
One can see that the point $[{\sqrt {-b}}:1:0] \in D_+(y) \subset \PP(3,2,1)$ corresponds to $(U,V,W)=(-b,0,0) \in \mathbb{A}_k^3$, 
and $J$ evaluated at that point is $J_1$, which has rank $2$; so $[\sqrt{-b}:1:0]$ is a regular point of $C$.

Thus, $O$ is the only singular point of $C$.
\end{proof} 

\begin{lemma}\label{lemma4} Let $D\in {\rm LND}(R)$ be irreducible and homogeneous, where $\krn D=k[g,h]$ for $g\in R_3$ and $h\in R_4$. 
Then, up to a homogeneous change of coordinates of $R$, $\krn D=k[x,xz+y^2]$ and $y$ is a minimal local slice for $D$.
\end{lemma}

\begin{proof} 
Since $g$ is irreducible and belongs to $R_3=kx\oplus kyz\oplus kz^3$,
$g=x+ayz+bz^3$ for some $a,b\in k$.
Since $h$ is irreducible and belongs to $R_4=kxz\oplus ky^2\oplus kyz^2\oplus kz^4 =kgz\oplus ky^2\oplus kyz^2\oplus kz^4$, we have
$h=gz+\alpha y^2+\beta yz^2+\gamma z^4$ for some $\alpha\in k^*$ and $\beta ,\gamma\in k$. We can rewrite this as
\begin{equation}\label{simple}
\text{$h=gz+\zeta (y + c_1 z^2) (y + c_2 z^2)$ for some $\zeta \in k^*$ and $c_1,c_2 \in k$.}
\end{equation}
Let $A=k[g,h]$ and $K = \Frac(A)$, and let $t \in R$ be any local slice of $D$.
Then {\it Lemma\,\ref{slice}} gives $R \subset K[t] = K^{[1]}$, so we may consider the $t$-degree, $\deg_t(r)$, of any $r\in R$.
Since $\deg(g) , \deg(h) > 2$, we have $R_2 \cap A = \{0\}$ and hence $y + c_i z^2 \notin A$ for $i=1,2$.
Thus, $y+c_1 z^2$ and $y+c_2 z^2$ are nonconstant polynomials in $K[t]$.
Since $g,h \in K$, \eqref{simple} implies
\[
\deg_t(z)= \deg_t(y+ c_1 z^2) + \deg_t(y+ c_2 z^2) ,
\]
which itself implies that $\deg_t(z) > \deg_t(y+ c_i z^2)$ for $i=1,2$.
Thus, 
\[
\deg_t(z) > \deg_t\big( (y+ c_1 z^2) - (y+ c_2 z^2) \big) = \deg_t \big( (c_1-c_2) z^2 \big)
\]
and consequently $c_1=c_2$.
Therefore,
\[
\text{$h=gz+\zeta (y+ c z^2)^2$ for some $\zeta\in k^*$ and $c \in k$.}
\]
Since $R=k[g,y+cz^2,z]$, we may assume (after a homogeneous change of coordinates of $R$) that $\krn D=k[x,xz+y^2]$.
By {\it Proposition\,\ref{jacobian}}, there exists $\lambda \in k^*$ such that
$D = \lambda \left| \frac{\partial(g,h, \, \underline{\ \ } \, )}{\partial(x,y,z)} \right| = \lambda (2y \frac{\partial}{\partial z} - x \frac{\partial}{\partial y})$,
and it is easy to check that $y$ is a minimal local slice for this derivation.
\end{proof}


\section{Proofs of the main theorems}

Throughout this section, $k$ is an algebraically closed field of characteristic $0$.

The first two lemmas are preparation for the proof of {\it Theorem\,\ref{main1}}.

\begin{setup}  \label {h82736f6wd8n912win} 
Let $R=\bigoplus_{d\in\N}R_d$ be the $\N$-grading of $R=k[x,y,z]\cong k^{[3]}$ where $x\in R_3$, $y\in R_2$ and $z\in R_1$.
Let $f\in R_9$ satisfy conditions (C1)-(C3). 
Assume that $D \in \lnd(R)$ is irreducible, homogeneous, and that $D^2(f)=0$.
Let $g,h$ be homogeneous elements such that $\krn D=k[g,h]$ and $\deg g \le \deg h$ (such $g,h$ exist by {\it Proposition\,\ref{Daigle00result19}}).
Define $d =\deg g$ and $e = \deg h$. By {\it Proposition\,\ref{Daigle00result19}}, $\deg (D)=d+e-6$ and $\gcd(d,e)=1$.
Since $R_1=kz$, $d$ and $e$ cannot both be 1, so $d<e$.
\end{setup}

\begin{lemma} \label {corollary}
Let the assumptions be as in Setup \ref{h82736f6wd8n912win}.
\begin{itemize}
\item [{\bf (a)}] $f$ is a local slice for $D$.
\item [{\bf (b)}] $z\notin\krn D$
\item [{\bf (c)}] Up to multiplication by an element of $k^*$, $g$ is equal to one of
$$
\text{$y+az^2$ for some $a\in k$, \quad or \quad  $x+ayz+bz^3$ for some $a,b\in k$.}
$$
In particular, $g$ is a variable of $R$.
\item [{\bf (d)}] $\rank(D) = 2$
\item [{\bf (e)}] $2\le d\le 3$ and $e\ge 5$.
\end{itemize}
\end{lemma}

\begin{proof}
Part (a). By hypothesis, $D^2(f)=0$.
Since $\Proj(R/fR)$ has at least two singular points by (C2), it is not a unicuspidal rational curve.
So $Df\ne 0$ by {\it Proposition\,\ref{8y23874t28746rdhjh}}, i.e., $f$ is a local slice of $D$.

Part (b). If $Dz=0$ then $D$ induces a locally nilpotent derivation $\bar D : R/zR \to R/zR$,
and $\bar D \neq 0$ because $D$ is irreducible.
Modulo $z$, we have $\bar{f}=x(px^2+qy^3)$ and $\bar{D}^2\bar{f}=0$.
Since $\deg_{\bar D}(px^2+qy^3) \le \deg_{\bar D}(x(px^2+qy^3)) \le 1$, $\bar{D}^2(px^2+qy^3)=0$. By {\it Proposition\,\ref{AB-theorem}}, it follows that $\bar{D}x=\bar{D}y=0$, 
contradicting the fact that $\bar D \neq 0$.
Therefore, $Dz\ne 0$. 

Part (c). {\it Proposition\, \ref{8y23874t28746rdhjh}} implies that $\delta f\ne 0$ for all nonzero $\delta\in {\rm LND}(R)$. 
Therefore, since $\deg (f)\ne 6$, {\it Lemma\, \ref{723476ed7dh0927gb}} implies that:
$$
d=\min(d,e) \le \max\{ 2,\deg f-(3+2+1)\}=3
$$
We have
\[
R_1=kz\,\, ,\,\,R_2=ky\oplus kz^2\,\, ,\,\, R_3=kx\oplus kyz\oplus kz^3
\]
as $k$-modules. 
If $g\in R_1$ then $g =\lambda z$ for some $\lambda\in k^*$, which contradicts part (b). So this case cannot occur. 
Since $d \le 3$, we have $g \in R_2$ or $g \in R_3$.
If $g\in R_2$, write $g =\lambda y+\mu z^2$ for $\lambda ,\mu\in k$, and if 
$g\in R_3$, write $g =\lambda x+\mu yz+\nu z^3$ for $\lambda ,\mu ,\nu\in k$. Since $g$ is prime we see that $\lambda\ne 0$ in both cases.
Multiplying $g$ by $\lambda^{-1}$ gives assertion (c).

Part (d). By part (c), $\ker(D)$ contains a variable of $R$ so $\rank(D)\le2$.
Proceeding by contradiction, assume that $\rank(D)\neq2$.
Then  $\rank(D)=1$. Since $D$ is irreducible and homogeneous, there exists $t\ge1$ and $s\in R_t$ such that $Ds=1$.
By {\it Proposition\,\ref{slice}} it follows that $R = k[g,h,s]$. If $t\ge 2$ then $R_1=\{ 0\}$, a contradiction since $R_1=kz$. Therefore, $t=1$ and $R = k[g,h,z]$.
Since $z$ is a slice and $f$ is a local slice for $D$, it follows that $\deg_z(f)=1$ when $f$ is expressed in the coordinates $k[g,h,z]$.
But then {\it Proposition\,\ref{curve-singularities}} implies that $\Proj(R/fR)$ has only one singularity, contradicting (C2). 
So ${\rm rank}(D)=2$ and part (d) is proved.

Part (e). We have $\gcd(d,e)=1$, and $2 \le d \le 3$ by part (c). Therefore, either $e\ge 4$ or $(d,e)=(2,3)$. 
In the latter case we have $g=ry+cz^2$ and $h=sx+ayz+bz^3$ for $a,b,c,r,s\in k$.
If $r=0$ or $s=0$ then $Dz=0$, since $k[g,h]$ is factorially closed in $R$, which would contradict part (b). So $r,s\in k^*$.
But then $R=k[g,h,z]$, so ${\rm rank}(D)=1$, contradicting part (d). 
So $e\ge 4$.

Suppose that $e=4$. Since $d$ and $e$ are relatively prime, this means $(d,e)=(3,4)$. By {\it Lemma\,\ref{lemma4}}, there exists a homogeneous minimal local slice $r\in R_2$.
By {\it Proposition\,\ref{plinth}}, there exist homogeneous $\alpha ,\beta\in k[g,h]$ such that $f=\alpha r+\beta$. 
Then $\alpha\in R_7\cap k[g,h]=k \cdot gh$ and $\beta\in R_9\cap k[g,h]=k\cdot g^3$ implies $f\in gR$, contradicting irreducibility of $f$. 
Therefore, $e\ge 5$, and part (e) is proved.
\end{proof}

\begin{lemma}\label{lemma3} 
Let the assumptions be as in Setup \ref{h82736f6wd8n912win}.
\begin{itemize}
\item [{\bf (a)}] If $r$ is a homogeneous minimal local slice for $D$ then $Dr \in k[g]$ and there exist $s,t\in k^*$ such that $f=shr+tg^3$. 
\item [{\bf (b)}] $(d,e) \in \{ (3,5), (3,7) \}$.
\end{itemize}
\end{lemma}

\begin{proof}
We have $Dg=0$ and, by {\it Lemma\,\ref{corollary}}, $g$ is a variable, $\rank(D)=2$, $d\in\{ 2,3\}$ and $e\ge 5$. 
We claim that
\begin{equation}  \label {uyrce3wlkpoi09zwxer}
Dr \in k[g] \quad \text{and} \quad Df \notin k[g].
\end{equation}
Indeed, {\it Lemma\,\ref{rank-two}} implies that the plinth ideal of $D$ is generated by a nonzero element of $k[g]$,
and Proposition \ref{plinth}(b) implies that $Dr$ generates the plinth ideal; so $Dr \in k[g]$.
The second part of \eqref{uyrce3wlkpoi09zwxer} follows from {\it Lemma\,\ref{723476ed7dh0927gb}(a)}.
Since $f$ is a local slice and $r$ is minimal, {\it Proposition\,\ref{plinth}} implies that there exist $\alpha ,\beta\in\krn D$ such that $f=\alpha r+\beta$. 
Since $f$ and $r$ are homogeneous, so are $\alpha$ and $\beta$.
If $\beta = 0$ then $\alpha \in k^*$ (because $f$ is irreducible) and hence $Df = \alpha Dr$ contradicts \eqref{uyrce3wlkpoi09zwxer}; so $\beta \neq 0$.
Since $k$ is algebraically closed, there exist factorizations
\[
\alpha = sg^ah^b\prod_{1\le i\le M}(s_ig^e+t_ih^d) \quad {\rm and}\quad
\beta = tg^{a'}h^{b'}\prod_{1\le j\le N}(s_i^{\prime}g^e+t_j^{\prime}h^d)
\]
where $a,b,a',b'\in\N$ and $s,t,s_i,t_i,s_j^{\prime},t_j^{\prime}\in k^*$.
By considering degrees we have:
\[
9=da+eb+Mde+\deg r=da'+eb'+Nde
\]
Since $d\ge 2$, $e\ge 5$ we have $de\ge 10$. Therefore, $M=N=0$ and
\[
f=sg^ah^br+tg^{a'}h^{b'} \quad \text{for some $a,b,a',b' \in \N$ and $s,t \in k^*$.}
\]
If $b=0$ then $Df = s g^a Dr$ contradicts  \eqref{uyrce3wlkpoi09zwxer}; so $b>0$.
Since $f$ is irreducible, $b'=0$; since $f$ is homogeneous, $a'>0$; since $f$ is irreducible, $a=0$.
The condition $d a' = 9$ implies that $d=3=a'$. Thus, $f=sh^br+tg^3$ where $b \ge 1$. 

If $\deg(r)=1$ then $r = \kappa z$ for some $\kappa\in k^*$. 
But then $f(x,y,0) = t g(x,y,0)^3$ is a cube, whereas (C3) implies that $f(x,y,0)$ is not a cube.
So $\deg(r)>1$.  Since $e\ge 5$ and $b e + \deg(r)=9$, we get $b=1$ and $e \le 7$.
Assertion (a) follows.
So $5 \le e \le 7$ (and $d=3$).
In view of the fact that $\gcd(d,e)=1$, part (b) follows.
\end{proof}

\begin{noname} {\bf Proof of {\it Theorem\,\ref{main1}}.}
Arguing by contradiction, suppose that $f\in R_9$ satisfies {\rm (C1)-(C3)} and $|f|_R < 2$.
Then there exists $D \in \lnd(R)$ such that $D \neq 0$ and $D^2(f)=0$.
It is well known that $D$ has a decomposition $D=\sum_{i=m}^nD_i$ for some $m\le n\in\N$,
where each $D_i$ is a homogeneous derivation of $R$ of degree $i$ and $D_n$ is nonzero and locally nilpotent;
see \cite{Freudenburg.17}, Proposition 3.8 and Principle 15.  Since $f$ is homogeneous, $D_n^2f=0$.
So we can choose $D$ to be homogeneous.
We can also choose $D$ to be irreducible, and by {\it Proposition\,\ref{Daigle00result19}} we can choose homogeneous $g,h$ such that $\ker(D) = k[g,h]$ and $\deg g \le \deg h$.
Let $d = \deg g$ and $e = \deg h$ (so all assumptions of Setup \ref{h82736f6wd8n912win} are satisfied).
By {\it Proposition\,\ref{plinth}}, there exists a minimal local slice $r$ of $D$; clearly, we can choose $r$ to be homogeneous.
By {\it Lemma\,\ref{lemma3}},
\[
(d,e) \in \{ (3,5), (3,7) \}  \quad \text{and} \quad f=shr+tg^3 \text{\ \ for some $s,t\in k^*$.} 
\]
Observe that $\deg (r)=9-e$. {\it Lemma\,\ref{lemma3}} also states that $Dr \in k[g]$, so we have $Dr =  \lambda g^i$ for some $\lambda \in k^*$ and $i \in \N$.
So {\it Proposition\,\ref{slice}}
implies that $R_g = (\krn D)_g[r] = k[g^{\pm1},r,h]$.
For any $j \in \N$, $g^j \frac{\partial}{\partial h} : k[g^{\pm1},r,h] \to k[g^{\pm1},r,h]$ is a locally nilpotent derivation of $R_g$ with kernel $k[g^{\pm1},r]$
and such that $(g^j \frac{\partial}{\partial h})^2(f) = 0$.
For a suitable choice of $j \in \N$, $g^j \frac{\partial}{\partial h}$ maps $R$ into itself.
For such a $j$, the restriction $\Delta : R \to R$ of  $g^j \frac{\partial}{\partial h}$  is a homogeneous locally nilpotent derivation of $R$ with
$\krn(\Delta) = k[g^{\pm1},r] \cap R$ and $\Delta^2(f)=0$.
Since $g$ and $r$ are homogeneous, and since $g$ is prime and $g \nmid r$ in $R$, we have $k[r] \cap gR = \{0\}$.
This implies that $k[g,r] \cap gR = gk[g,r]$, which in turn implies that $k[g^{\pm1},r] \cap R = k[g,r]$.  Thus, $\krn(\Delta) = k[g,r]$.
We have $\Delta = a \delta$ for some $a \in \ker(\Delta) \setminus \{0\}$ and some irreducible $\delta \in \lnd(R)$.
Then $\delta$ is homogeneous and satisfies $\ker(\delta) = k[g,r]$ and $\delta^2(f)=0$, so $\delta$ satisfies all assumptions of Setup \ref{h82736f6wd8n912win}.
Since $\deg g = 3$ and $\deg r = 9-e$, applying {\it Lemma\,\ref{lemma3}} to $\delta$ implies that 
one of the pairs $(3,9-e), (9-e,3)$ belongs to $\{(3,5), (3,7)\}$.
Since $e \in \{5,7\}$, this is impossible.
This completes the proof of  {\it Theorem\,\ref{main1}}. \hfill $\qed$
\end{noname}


\begin{noname} {\bf Proof of {\it Theorem\,\ref{main2}}.}
We have that $R=k^{[3]}$ is a $\Z$-graded UFD, $f\in R$ is a homogeneous prime, and $\deg f=9$. In addition, by {\it Theorem\,\ref{main1}}, we have $|f|_R\ge 2$.
{\it Proposition\,\ref{FMJ13}} implies that $R[W]/(W^n+f)$ is a rigid rational UFD. \hfill $\qed$
\end{noname}


\subsection*{Preparation for the proof of {\it Theorem\,\ref{main3}}}

As above, let $R=k[x,y,z]=k^{[3]}$ with $\N$-grading $R=\bigoplus_{i\in\N}R_i$ where $x\in R_3$, $y\in R_2$ and $z\in R_1$.
Let $\ggoth$ be this $\N$-grading and consider the subtorus $\Gamma = T(\ggoth)$ of $\Aut_k(R)$
(see paragraph \ref{7t65fr56gb12dwe8} and Definition \ref{lkjmnbv987y123erfpoiusdfv8gut}). Let $\CC(\Gamma)$ be the centralizer of $\Gamma$ in $\Aut_k(R)$.

\begin{lemma}\label{centralizer1} Let $h\in R_9$ satisfy conditions {\rm (C3)} and {\rm (C4)}. 
If $\alpha\in\mathcal{C}(\Gamma)$ and $\alpha (h)=h$ then $\alpha\in \Gamma$. 
\end{lemma}

\begin{proof}
By \eqref{mnbkjhgertypoi345098asdb7h2wsgt}, the hypothesis $\alpha\in\CC(\Gamma)$ implies that $\alpha(R_i) = R_i$ for all $i \in \N$, so
\[
\alpha (x)\in R_3=kx\oplus kyz\oplus kz^3\,\, ,\,\, \alpha (y)\in R_2=ky\oplus kz^2\,\, ,\,\, \alpha (z)\in R_1=kz
\]
Let $\kappa ,\lambda ,\mu\in k^*$ and $a,b,c\in k$ be such that:
\[
\alpha(x) = \kappa x+ayz+bz^3 , \quad
\alpha(y) = \lambda y+cz^2 , \quad
\alpha(z) = \mu z .
\]
Write $h=\sum_{0\le i\le 9}h_iz^i$ for homogeneous $h_i\in k[x,y]$. Observe that:
\[
h_0\in kx^3\oplus kxy^3 \,\, ,\,\, h_1\in kx^2y\oplus ky^4 \,\, ,\,\, h_2\in kxy^2\,\, ,\,\, h_3\in kx^2\oplus ky^3
\]
We have $h_0=x(px^2+qy^3)$ for $p,q\in k^*$ and $h_1=y(rx^2+sy^3)$ for $r,s\in k$ where $\det\left(\begin{smallmatrix} 3p & r\cr q & s \end{smallmatrix}\right)\ne 0$.
Let $t,u,v\in k$ be such that $h_2=txy^2$ and $h_3=ux^2+vy^3$.
Let $h_i'(x,y) \in k[x,y]$ denote the coefficient of $z^i$ in $\alpha(h) = \sum_{i=0}^9 h_i( \alpha(x), \alpha(y) ) \alpha(z)^i$.
Since $\alpha(h)=h$, we have $h_i'(x,y) = h_i(x,y)$ for all $i=0,\dots,9$.
The following are the equalities  $h_i'(x,y) = h_i(x,y)$ for $i=0,1,2,3$:
\begin{gather*}
(\kappa^3p)x^3+(\kappa\lambda^3q)xy^3  =px^3+qxy^3 , \\
(\kappa^2\lambda\mu r+3a\kappa^2p)x^2y+(\lambda^4\mu s+a\lambda^3)qy^4  = rx^2y+sy^4 , \\
(\kappa\lambda^2\mu^2t+2a\kappa\lambda\mu r+3c\kappa\lambda^2q+3a^2\kappa p)xy^2  =txy^2 ,  \\
\begin{split}
(\kappa^2\mu^3u+c\kappa^2\mu r+3b\kappa^2p)x^2+(\lambda^3\mu^3v+a\lambda^2\mu^2t+4c\lambda^3\mu s+a^2\lambda\mu r+3ac\lambda^2q+b & \lambda^3q+a^3p)y^3 \\ &= ux^2+vy^3 .
\end{split}
\end{gather*}
The first of these equalities gives $\kappa^3=\kappa\lambda^3=1$. From the second, we find that:
\[
\kappa^2\lambda\mu r+3a\kappa^2p=r \quad {\rm and}\quad \lambda^4\mu s+a\lambda^3 q =s .
\]
Multiplying each by $\kappa$ gives $\lambda\mu r+3ap=\kappa r$ and $\lambda\mu s+aq=\kappa s$, which yields:
\[
\begin{pmatrix} 3p & r\cr q&s\end{pmatrix} \begin{pmatrix} a \\ \lambda\mu -\kappa \end{pmatrix}
= \begin{pmatrix} 0\cr 0\end{pmatrix}
\]
Since the determinant of the square matrix is nonzero, we conclude that $a=0$ and $\lambda\mu=\kappa$. 
Setting $a=0$ in the third equality gives $\kappa\lambda^2\mu^2t+3c\kappa\lambda^2q=t$. Since $\kappa\lambda^2\mu^2=1$ we see that $c=0$. 
Setting $a=0=c$ in the fourth equality gives $\lambda^3\mu^3v+b\lambda^3q=v$. Since $\lambda^3\mu^3=1$ we see that $b=0$.

Therefore, $\alpha(x) =\kappa x$, $\alpha(y) = \lambda y$ and $\alpha(z) =\mu z$. Since $\kappa^3=1$, $\lambda^3=\kappa^2$ and $\lambda\mu =\kappa$, it follows that
$\mu^2=\kappa^2\lambda^{-2}=\lambda^3\lambda^{-2}=\lambda$ and $\mu^3=\kappa^3\lambda^{-3}=\kappa^{-2}=\kappa$, so
\[
\alpha(x) =\mu^3 x, \quad \alpha(y) = \mu^2 y, \quad \alpha(z) =\mu z ,
\]
which means that $\alpha \in \Gamma$.
\end{proof}

Next, let $f\in R_9$ satisfy conditions (C1)-(C4). 
Given $n\in\Z$ with $n\ge 2$ and $n\notin 3\Z$ define
\[
A=R[W]/(W^n-f)=R[w]
\]
where $R[W]=R^{[1]}$. By {\it Theorem\,\ref{main2}}, $A$ is a rigid rational affine UFD. 

The symbol $\mathfrak{g}$ continues to denote the aforementioned $\N$-grading of $R$.
Let $\Ggoth$ denote the $\N$-grading of $A$, $A=\bigoplus_{i\in\N}A_i$, obtained by
declaring that $x \in A_{3n}$, $y \in A_{2n}$, $z \in A_{n}$ and $w \in A_{9}$.
Consider the subtorus $T = T(\Ggoth)$ of $\Aut_k(A)$ and let $\mathcal{C}(T)$ be the centralizer of $T$ in ${\rm Aut}_k(A)$. 
Since $T$ is abelian, $T\subseteq\mathcal{C}(T)$.

\begin{lemma}\label{centralizer2} $\mathcal{C}(T)=T$
\end{lemma}

\begin{proof} Let $\alpha\in\mathcal{C}(T)$. By \eqref{mnbkjhgertypoi345098asdb7h2wsgt}, we have $\alpha(A_i) = A_i$ for all $i \in \N$.
Since $A_9=kw$, there exists $\tau\in T$ so that $\tau\alpha (w)=w$. 
Note that $\tau\alpha\in\mathcal{C}(T)$ and that it suffices to show that $\tau\alpha \in T$.
So we may assume, with no loss of generality, that $\alpha (w)=w$.
Observe that $R_i = A_{ni}$ for all $i \in \N$.
So $\alpha(R) = R$ and we may consider $\alpha\vert_R \in \Aut_k(R)$.
Since $\alpha\vert_R(R_i) = R_i$ for all $i \in \N$, we have $\alpha\vert_R\in\mathcal{C}(\Gamma)$ by \eqref{mnbkjhgertypoi345098asdb7h2wsgt}.
Since $f=w^n =\alpha (w)^n =\alpha (f)$, it follows by {\it Lemma\,\ref{centralizer1}} that $\alpha\vert_R\in\Gamma$. 
So there exists $\mu \in k^*$ such that $\alpha(u) = \mu^i u$ for all $u \in R_i$ and all $i \in \N$ (in particular, $f = \alpha(f) = \mu^9 f$, so $\mu^9=1$).
Pick $\lambda_1 \in k^*$ such that $\lambda_1^n = \mu$; then  $\alpha(u) = \lambda_1^{ni} u$ for all $u \in A_{ni}$ and all $i \in \N$.
Observe that 
\[
\left( \frac{ \alpha(w) }{ \lambda_1^9 w } \right)^n
= \frac{ \alpha(w^n) }{ \lambda_1^{9n} w^n } 
= \frac{ \alpha(f) }{ \mu^9 f } = 1
\]
so $\frac{ \alpha(w) }{ \lambda_1^9 w } \in \Omega_n$ where $\Omega_n = \setspec{ \zeta \in k^* }{ \zeta^n=1 }$.
Since $\zeta \mapsto \zeta^9$ is a bijection $\Omega_n \to \Omega_n$, there exists $\zeta \in \Omega_n$ such that $\zeta^9 = \frac{ \alpha(w) }{ \lambda_1^9 w }$.
Let $\lambda = \zeta \lambda_1$. Then
$$
\text{$\alpha(w) = \lambda^9 w$, and for every $i \in \N$ and $u \in A_{ni}$, $\alpha(u) = \lambda^{ni} u$.}
$$
Since $A = R[w]$, it follows that $\alpha(u) = \lambda^j u$ for all $u \in A_j$ and $j \in \N$, i.e., $\alpha \in T$.
\end{proof}

\begin{noname} {\bf Proof of {\it Theorem\,\ref{main3}}.}
We have $\CC(T)=T$ by Lemma \ref{centralizer2}.
Since $A$ is rigid, Corollary \ref{5cFt2xwdlk12pwo09284x2ednBb7w36gt1} implies that $\Aut_k(A) = T$.  \hfill $\qed$
\end{noname}

\begin{noname} \label {Proofmain4}  {\bf Proof of {\it Corollary\,\ref{main4}}.}
The following assertion immediately follows from the corollary stated in the introduction of \cite{Kurth.97}:
\begin{equation}  \label {jhd827u36erfwq8hx}
\textit{Assume that $k = \C$ and let $n\ge1$.  If $\Aut_\C( Q_n ) \cong \C^*$ then  $\Aut_{\SL_2}( V_n ) \cong \C^*$.}
\end{equation}
By \eqref{jhd827u36erfwq8hx} and Corollary \ref{cor1}, we obtain that if $k=\C$ then $\Aut_{\SL_2}( V_5 ) \cong \C^*$.
Let us now show that this remains true without the assumption $k=\C$.
We will write ``$k \in \ACF$'' as an abbreviation for ``$k$ is an algebraically closed field of characteristic zero''.
If $k \in \ACF$, let $V_n^{(k)} \subset k[x,y]$ be the $k$-vector space of homogeneous polynomials of degree $n$
and consider the left-action (defined in the introduction) of $\SL_2(k)$ on the variety $V_n^{(k)} \cong \A^{n+1}_k$. 
We know that  $\Aut_{\SL_2(\C)}( V_n^{(\C)} ) \cong \C^*$ is true for $1 \le n \le 5$.
Let us prove the following claim:
\begin{equation}  \label {udb998ub239dn28ho}
\begin{minipage}[t]{.70\textwidth}
\it
If $n\ge1$ is such that $\Aut_{\SL_2(\C)}( V_n^{(\C)} ) \cong \C^*$,
then $\Aut_{\SL_2(k)}( V_n^{(k)} ) \cong k^*$ for all $k \in \ACF$.
\end{minipage}
\end{equation}
Indeed, let $k \in \ACF$. The left-action of $\SL_2(k)$ on the variety $V_n^{(k)}$ induces
a right-action of $\SL_2(k)$ on the coordinate algebra $k[ V_n^{(k)} ] \cong k^{[n+1]}$,
and we have an anti-isomorphism of groups $\Aut_{\SL_2(k)}( V_n^{(k)} ) \to \Aut_{\SL_2(k)}( k[V_n^{(k)}] )$, $\varphi \mapsto \varphi^*$.
Let $\ggoth_k$ denote the standard $\N$-grading of $k[V_n]$ and consider the subtorus $T(\ggoth_k) \cong k^*$ of $\Aut_k( k[V_n^{(k)}] )$.
Then  $T(\ggoth_k) \subseteq \Aut_{\SL_2(k)}( k[V_n^{(k)}] )$, and consequently
\begin{equation} \label {o8ih2g73trgefd}
\Aut_{\SL_2(k)}( V_n^{(k)} ) \ncong k^* \iff  \text{the inclusion $T(\ggoth_k) \subset \Aut_{\SL_2(k)}( k[V_n^{(k)}] )$ is strict.}
\end{equation}
Also, it is an easy exercise to show:
\begin{equation} \label {6c5exs3wdfex1g2h8wd}
\begin{minipage}{.8\textwidth}
Given $\phi \in \Aut_k( k[V_n^{(k)}] )$, we have $\phi \notin T(\ggoth_k)$ if and only if there exists a homogeneous $f \in k[ V_n^{(k)} ]$
such that $\phi(f) \notin k f$.
\end{minipage}
\end{equation}
Fix $n\ge1$ and let $\Kscr(n)$ be the class of fields $k \in \ACF$ satisfying  $\Aut_{\SL_2(k)}( V_n^{(k)} ) \ncong k^*$.
To prove \eqref{udb998ub239dn28ho}, it suffices to show that if $\Kscr(n) \neq \emptyset$ then $\C \in \Kscr(n)$.
Consider $k \in \Kscr(n)$.
By \eqref{o8ih2g73trgefd}, there exists $\phi \in \Aut_{\SL_2(k)}( k[V_n^{(k)}] ) \setminus T(\ggoth_k)$;
by \eqref{6c5exs3wdfex1g2h8wd},  there exists a homogeneous $f \in k[ V_n^{(k)} ]$ such that $\phi(f) \notin k f$.
Choose $k_0 \in \ACF$ such that $k_0 \subseteq k$, $\trdeg_\Q(k_0) < \infty$, $\phi$ restricts to an automorphism $\phi_0$ of $k_0[ V_n^{(k_0)} ]$,
and $f \in k_0[ V_n^{(k_0)} ]$. Since $\phi$ is $\SL_2(k)$-equivariant, $\phi_0$ is $\SL_2(k_0)$-equivariant, i.e., $\phi_0 \in \Aut_{\SL_2(k_0)}( k[V_n^{(k_0)}] )$.
Since $\phi(f) \notin k f$, we have $\phi_0(f) \notin k_0 f$, so $\phi_0 \notin T(\ggoth_{k_0})$ by \eqref{6c5exs3wdfex1g2h8wd},
so $\Aut_{\SL_2(k_0)}( V_n^{(k_0)} ) \ncong k_0^*$ by \eqref{o8ih2g73trgefd}. Since $k_0$ is isomorphic to a subfield of $\C$, we have shown 
that {\it if $\Kscr(n) \neq \emptyset$ then some subfield of $\C$ belongs to $\Kscr(n)$.}

Next, consider a subfield $k$ of $\C$ such that $k \in \Kscr(n)$.
Let $\bE : \Scho k \to \Scho{\C}$ be the base extension functor from the category of schemes over $k$ to the category of schemes over $\C$.
Clearly, $\bE( V_n^{(k)} ) = V_n^{(\C)}$.
It is a general property of group schemes that $\bE( \SL_2(k) ) = \SL_2(\C)$ and that if $\varphi : V_n^{(k)} \to V_n^{(k)}$ is $\SL_2(k)$-equivariant
then $\bE( \varphi ) : V_n^{(\C)} \to V_n^{(\C)}$ is $\SL_2(\C)$-equivariant.  So
\begin{equation} \label {iucytf6172usnx}
\text{$\varphi \mapsto \bE(\varphi)$ is a group homomorphism from $\Aut_{\SL_2(k)}( V_n^{(k)} )$ to $\Aut_{\SL_2(\C)}( V_n^{(\C)} )$.}
\end{equation}
Let $\Escr$ denote the functor $\C \otimes_k (\underline{\ \ })$ from the category of $k$-algebras to the category of $\C$-algebras.
Then $\Escr( k[V_n^{(k)}] ) = \C[V_n^{(\C)}]$ and 
$\phi \mapsto \Escr(\phi)$ is a group homomorphism from $\Aut_{k}( k[V_n^{(k)}] )$ to $\Aut_{\C}( \C[V_n^{(\C)}] )$.
By \eqref{iucytf6172usnx}, $\phi \mapsto \Escr(\phi)$ preserves $\SL_2$-equivariance, so
$$
\text{$\phi \mapsto \Escr(\phi)$ is a group homomorphism from $\Aut_{\SL_2(k)}( k[V_n^{(k)}] )$ to $\Aut_{\SL_2(\C)}( \C[V_n^{(\C)}] )$.}
$$
Since $k \in \Kscr(n)$, \eqref{o8ih2g73trgefd} implies that we can choose $\phi \in \Aut_{\SL_2(k)}( k[V_n^{(k)}] ) \setminus T(\ggoth_k)$;
by \eqref{6c5exs3wdfex1g2h8wd},  there exists a homogeneous $f \in k[ V_n^{(k)} ]$ such that $\phi(f) \notin k f$.
Consider the element $\tilde\phi = \Escr(\phi)$ of $\Aut_{\SL_2(\C)}( \C[V_n^{(\C)}] )$
and the canonical homomorphism $\nu : k[V_n^{(k)}] \to \C \otimes_k k[V_n^{(k)}] = \C[V_n^{(\C)}]$, $\nu(x) = 1 \otimes x$.
The following diagram commutes:
$$
\scalebox{.85}{\xymatrix{
\C[V_n^{(\C)}] \ar[r]^-{\tilde\phi}  &  \C[V_n^{(\C)}]  \\
k[V_n^{(k)}] \ar[u]^-{\nu} \ar[r]^-{\phi} &  k[V_n^{(k)}] \ar[u]_-{\nu}
}}
$$
Since $f$ and $\phi(f)$ are linearly independent over $k$,
it follows that $\nu(f)$ and $\nu(\phi(f)) = \tilde\phi(\nu(f))$ are linearly independent over $\C$,
so $\tilde\phi( \nu(f) ) \notin \C \cdot \nu(f)$.
By \eqref{6c5exs3wdfex1g2h8wd}, this implies that $\tilde\phi \notin T(\ggoth_\C)$,
so \eqref{o8ih2g73trgefd} implies that  $\Aut_{\SL_2(\C)}( V_n^{(\C)} ) \ncong \C^*$,
i.e., $\C \in \Kscr(n)$.

This proves \eqref{udb998ub239dn28ho}.
It follows that $\Aut_{\SL_2(k)}( V_n^{(k)} ) \cong k^*$ for all $n \le 5$ and $k \in \ACF$, so in particular {\it Corollary\,\ref{main4}} is true. \hfill\qedsymbol
\end{noname}


\section{Remarks on The binary sextic}
Let $\mathcal{Q}$ be the ring of invariants for the irreducible representation of $\SL_2(k)$ of degree 6 (the binary sextic representation). 
Then $\mathcal{Q}$ is a rational affine UFD and $\dim \mathcal{Q}=4$.
As with the binary quintic, ${\rm Spec}(\mathcal{Q})$ is a hypersurface. In \cite{Salmon.1885}, Articles 252 and 260, 
Salmon shows that $\mathcal{Q}$ is generated by homogeneous invariants $A,B,C,\Delta$  (=discriminant) and $E$, where $R:=k[A,B,C,\Delta ]\cong k^{[4]}$ and $E^2\in k[A,B,C,\Delta ]$, 
and $\deg (A,B,C,\Delta ,E)=(2,4,6,10,15)$. Let $R[\lambda ]=R^{[1]}$ and define $S,T\in R[\lambda ]$ by 
\[
S=1024\,\lambda^3-1152\,A\lambda^2+(132\,A^2-10800\,B)\lambda +(3375\,C+2700\,AB-4A^3)
\]
and:
\[
T=\lambda(256\lambda^2-320\,A\lambda+55\,A^2+4500\,B)^2-\Delta
\]
Salmon showed that (up to a scalar multiple) $E^2={\rm Res}_{\lambda}(S,T)$, where ${\rm Res}_{\lambda}(S,T)$ is the resultant of $S$ and $T$.
We calculated the polynomial ${\rm Res}_{\lambda}(S,T)$ using {\it Maple}. 
While the number of terms in this resultant is not too large, the coefficients are very large integers, and this remains true after removing their greatest common divisor. 
We proceed to make some simplifications.

Define:
\[
B'=2^23^25^2B+11A^2 \quad {\rm and}\quad C'=3^35^3C+2^23^35^2AB-2^2A^3
\]
Then 
\[
S=2^{10}\lambda^3+2^73^2A\lambda^2+2^23(B'-2\cdot 11A^2)\lambda +C'
\]
and:
\[
T=\lambda(2^8\lambda^2-2^65A\lambda +5B')^2-\Delta
\]
Replacing $\lambda$ by $2^{-4}\lambda$ we obtain
\[
2^2S=\lambda^3+2\cdot 3^2A\lambda^2+3(B'-2\cdot 11A^2)\lambda+2^2C'
\]
and:
\[
2^4T=\lambda(\lambda^2-2^25A\lambda +5B')^2-2^4\Delta
\]
If we now use $(u,x,y,z)=(2^4\Delta,2^2C',B',A)$ with the grading $\deg (u,x,y,z)=(5,3,2,1)$ we obtain:
\[
S=\lambda^3+(18z)\lambda^2+3(y-22z^2)\lambda+x \quad {\rm and}\quad T=\lambda(\lambda^2-(20z)\lambda+5y)^2-u
\]
In this form, we find that $P:={\rm Res}_{\lambda}(S,T)$ equals:
\begin{multline*}
138468423456uz^{10} - 52450160400xy^2z^8 - 42353126640uyz^8 + 6357595200x^2yz^7 \\
+ 13658752800xy^3z^6 - 1486797840uxz^7 + 3962439360uy^2z^6 - 192654400x^3z^6 - 1020180000x^2y^2z^5 \\
- 880071600xy^4z^4 + 217534440uxyz^5 - 102459600uy^3z^4 
+ 12117240x^3yz^4 + 24493600x^2y^3z^3 \\- 1192800xy^5z^2 
- 15933528u^2z^5 - 3156780ux^2z^4 + 5301960uxy^2z^3 
- 142080uy^4z^2 - 27760x^4z^3 \\- 236040x^3y^2z^2 + 16800x^2y^4z 
- 400xy^6 + 189960u^2yz^3 - 53940ux^2yz^2 + 9240uxy^3z 
- 48uy^5 + 840x^4yz \\ - 40x^3y^3 - 6030u^2xz^2  + 840u^2y^2z 
+ 210ux^3z - 60ux^2y^2 - x^5 - 15u^2xy - u^3
\end{multline*}
Note that $P$ is a homogeneous polynomial of degree 15 with 36 terms, having degree 10 in $z$, where $P\in (u,x)$ and $P(u,x,0,0)=-(x^5+u^3)$.
We have:
\[
\mathcal{Q}\cong k[u,x,y,z,v]/(v^2-P(u,x,y,z))
\]
By {\it Proposition\,\ref{FMJ13}}, $\mathcal{Q}$ is rigid if and only if $|P|_R \ge2$.



\bigskip

\noindent \address{Department of Mathematics and Statistics\\
University of Ottawa\\
Ottawa, ON, Canada, K1N 6N5\\
\email{ddaigle@uottawa.ca}
\bigskip

\noindent \address{Department of Mathematics\\
Western Michigan University\\
Kalamazoo, MI 49008}\\
USA\\
\email{gene.freudenburg@wmich.edu}

\end{document}